\documentclass[12pt]{amsart}
\usepackage{graphicx}
\usepackage{amssymb}
\usepackage{hyperref}
\usepackage{verbatim}
\usepackage[english]{babel}    
\usepackage{mathpazo} % math & rm
\usepackage[scaled]{helvet} % ss
\usepackage{courier} % tt
\normalfont
\usepackage[T1]{fontenc}
\newtheorem{theorem}{Theorem}[section]
\newtheorem{corollary}[theorem]{Corollary}
\newtheorem{lemma}[theorem]{Lemma}
\newtheorem{prop}[theorem]{Proposition}
\theoremstyle{remark}
\newtheorem{rem}[theorem]{\bf Remark}
\theoremstyle{definition}
\newtheorem{definition}[theorem]{Definition}
\theoremstyle{definition}
\newtheorem*{ack}{Acknowledgments}
\theoremstyle{remark}
\newtheorem{example}[theorem]{\bf Example}
\theoremstyle{remark}

\newcommand{\dbar}{d\llap{\raise 0.68ex\hbox{-}}}
\newcommand{\im}{{\rm im}}

\newcommand{\trg}{{\rm tr}_G}
\newcommand{\Trg}{{\rm Tr}_G}

\newcommand{\dom}{{\rm dom}}

\newcommand{\bid}{\, {\bf 1}}
\newcommand{\Hmm}[1]{\leavevmode{\marginpar{\tiny%
$\hbox to 0mm{\hspace*{-0.5mm}$\leftarrow$\hss}%
\vcenter{\vrule depth 0.1mm height 0.1mm width \the\marginparwidth}%
\hbox to 0mm{\hss$\rightarrow$\hspace*{-0.5mm}}$\\\relax\raggedright #1}}}
\begin{document}
\title[Pseudodifferential operators]{Generalized Fredholm properties for invariant pseudodifferential operators}
\author{Joe J Perez\ \ }
\address{Fakult\"at f\"ur Mathematik\\
Universit\"at Wien\\
Vienna, Austria}
\email{joe\_j\_perez@yahoo.com}
\thanks{JJP is supported by FWF grants P19667 and I382}


\begin{abstract} We define classes of pseudodifferential operators on $G$-bundles with compact base and give a generalized $L^2$ Fredholm theory for invariant operators in these classes in terms of von Neumann's $G$-dimension. We combine this formalism with a generalized Paley-Wiener theorem, valid for bundles with unimodular structure groups, to provide solvability criteria for invariant operators. This formalism also gives a basis for a $G$-index for these operators. We also define and describe a transversal dimension and its corresponding Fredholm theory in terms of anisotropic Sobolev estimates, valid also for similar bundles with nonunimodular structure group.\end{abstract}

\maketitle
{\tiny\tableofcontents}

%%%%%%%%%%%%%%%%%%%%
\section{Introduction}
%%%%%%%%%%%%%%%%%%%%
%
We will discuss an $L^2$ theory of some classes of pseudodifferential operators on manifolds $M$ as follow. Our $M$ will always be total spaces of $G$-bundles
\[G\longrightarrow M\stackrel \pi \longrightarrow X,\]
\noindent
with $G$ connected, usually unimodular Lie groups and $X$ compact manifolds. 

In this paper we describe some natural classes of pseudodifferential operators on $M$ and analyze the solvability of $G$-invariant operators in those classes. Our method will be a generalized Fredholm property due to M.\ Breuer \cite{B} as applied in \cite{A} by M.\ Atiyah and by A.\ Connes and H.\ Moscovici in \cite{CM}. This Fredholm property is based on a generalized idea of the dimension of a vector space due to J.\ von Neumann. This dimension, $\dim_G$, is defined for closed, $G$-invariant subspaces of Hilbert spaces on which a unimodular group $G$ acts.

In order to use this, we will construct natural Hilbert spaces $L^2(M)$ and Sobolev spaces $H^s(M)$ of (sections of bundles over) $M$ on which the $G$-action is strongly continuous and unitary. This allows us to define a trace $\Trg$ in the algebra $\mathcal B(L^2(M))^G$ of bounded operators in $L^2(M)$ commuting with the action of $G$. Applying this trace to orthogonal projections $P_L$ onto $G$-invariant subspaces $L\subset L^2(M)$ provides a dimension function $\dim_G$ given by 
\[\dim_G(L) = \Trg(P_L).\]
Roughly speaking, the generalized Fredholm property and index are then defined as usual, but in terms of this dimension. 

In this paper, many results will follow from the following technical fact relating the Sobolev degree to the trace class.
\begin{prop}\label{bigprop} Let $n=\dim M$. If $s>n/2$ and $A\in\mathcal B(L^2(M))^G$ has $\im(A)\subset H^s(M)$, then $\Trg(A^*A)<\infty$.\end{prop}
Defining $UL_G^m(M)$ to be the class of $G$-invariant pseudodifferential operators uniformly in the H\"ormander class $L^m$ on $M$, the above proposition, together with properties of these operator classes, will give that if $A\in UL_G^m(M)$ with $m<-n/2$, then $\Trg(A^*A)<\infty$. These results reduce to well-known optimal conditions put forth to obtain membership in the Hilbert-Schmidt class when $M$ is compact, \cite[\S8]{Sh}, and in the $\Gamma$-Hilbert-Schmidt operators when $M$ has a cocompact discrete group action, \cite[Thm.\ 3.4]{Schi}. We will also obtain that for $m< -n$, $A\in UL_G^m(M)$ is a $G$-trace-class operator, generalizing the results for the compact and cocompact discrete $\Gamma$ cases.
An easy consequence will be
\begin{corollary}\label{easy} Let $G$ be a connected unimodular Lie group and suppose that $M$ is the total space of a $G$-bundle with compact base. It follows that if $A\in UL_G^m(M)$ is elliptic, then the $G$-index 
\[{\rm ind}_G(A) = \dim_G\ker(A) - \dim_G\ker(A^*) \]
of $A$ is defined.\end{corollary}
We will also set up the $G$-Fredholm theory for elliptic operators $EUL_G^m(M)$ in $UL_G^m(M)$ and describe solvability in $L^2$ with Sobolev estimates. Since $M$ possesses a global $G$-action, it makes sense to convolve functions $f$ on $M$ by kernels $\kappa$ on $G$. We denote this by $f\mapsto \rho_\kappa f$. The solvability statement is
\begin{theorem} If $m\ge 1$, $A\in EUL_G^m(M)$ is self-adjoint, and $f\in C^\infty_c(M)$, then we may solve $Au=g$ in $L^2$, with the uniform Sobolev estimates $\|u\|_{s+m}\lesssim\|g\|_s$, for all $g=\rho_\kappa f$ in a space of infinite $G$-dimension in $L^2(M)$. Furthermore, all such $g$ correspond to convolution kernels $\kappa$ belonging to $C^\infty\cap L^2(G)$. Put differently, $Au=f$ is not only solvable in $L^2$ modulo errors in $H^\infty$; but also the equation can be solved exactly in an infinite-$G$-dimensional space consisting of smooth convolutions of $f$ itself. Furthermore, the kernels of $A$ and of any invariant parametrix of $A$ contain no elements of $L^2$ with compact support.\end{theorem}
\begin{rem} All of our results extend trivially to their analogues in Hermitian $G$-vector-bundles over $M$.\end{rem}
The proof of the theorem depends on a generalized Paley--Wiener theorem, valid for $G$-bundles, which combines finely with the $G$-Fredholm property. This method was developed in \cite{P2} based on a fundamental theorem of D.\ Arnal and J.\ Ludwig, \cite{AL}.

Let us now discuss the pseudodifferential calculi we will be using. L.\ H\"ormander in \cite{H1} defined the classes of symbols $S^m$ on manifolds and R.\ Strichartz began in \cite{S} the study of invariant pseudodifferential operators on Lie groups. In \cite{MS1, MS2}, G.\ Meladze and M.\ Shubin set up pseudodifferential calculi of uniform (but not necessarily invariant) operators on unimodular Lie groups. These were based on the classes $S^m$ of H\"ormander, particularly exploiting the existence of available metric- and measure-theoretic invariances on such groups. In \cite{MS3} other classes on unimodular groups are discussed which allow for the construction of complex powers and provide estimates on the Green function; the entire theory on groups is reviewed succinctly in \cite{MS1}. In \cite{Ko1, Ko2}, Yu.\ Kordyukov took this work of Meladze and Shubin further, generalizing to richer classes of operators on manifolds which do not possess an exact invariance but are of bounded geometry. 
 
Here we will take the properly supported operators of Kordyukov in \cite{Ko1} and apply the treatment there to the situation of $M$ as above. We will then consider subclasses of invariant, properly supported operators on $M$. The passage to invariant operators will use a small modification of an averaging method from \cite{CM}, which maps some general operators to invariant ones. It happens that these operators have good extensions to $L^2$ and we will derive sufficient conditions for membership of these extended operators in generalized trace and Fredholm classes as in \cite{P1}. Our results here should easily extend to the setting of \cite{Schi}.

Work related to ours is in a series of papers of  V.\ Nistor, E.\ Troitsky, A.\ Weinstein, and Ping Xu; \cite{N1, N2, NT, NWX}. Here, the authors constructed and applied an index theory on \emph{families} of Lie groups. Families are more general than our bundles as the fiber is allowed to vary along the base $X$, however, in their work, different assumptions are placed on the type of fiber. For example, in \cite{N1}, all the fibers are assumed to be simply connected and solvable and in \cite{NT}, the fibers are assumed to be compact Lie groups. Their technique does not require that the groups be unimodular, and in \cite{NWX} they even drop the requirement that the fiber be a group.

As in our case, the family encodes the symmetries of an elliptic operator on a bundle with the same base.  The aim in \cite{N1} is a formula for the Chern character of the gauge-equivariant index, similar to the Atiyah--Singer index formula for families, however it also includes information on the topology of the family of Lie groups that is considered. In our work as well as in theirs, of course, one obtains existence theorems for invariant pseudodifferential operator equations, though by substantially different means. 

Actions which are not free lead to other complications, as studied by P.\ Albin, R.\ Mazzeo, R.\ Melrose, and others; see \cite{AM}.

The contents of this paper are as follows. In Sect.\ \ref{msc} the uniform classes of pseudodifferential operators on $M$ will be defined, giving their principal properties. As we have said, this is closely related to \cite{Ko1}. Sect.\ \ref{HMT} is a description of the relevant Hilbert spaces over $M$ which allow us to bring techniques of von Neumann algebras to bear on invariant problems. In Sect.\ \ref{gtraceandmsc}, we relate membership in the $G$-Hilbert-Schmidt and $G$-trace classes to the uniform classes. Sect.\ \ref{PW} contains the existence theory of $G$-Fredholm operators in terms of the generalized Paley--Wiener theorem and the proof of the main theorem.
%
%%%%%%%%%%%%%%%%%%%%
\section{Uniform classes of pseudodifferential operators}\label{msc}
%%%%%%%%%%%%%%%%%%%%
%
\subsection{Local estimates for operators} Though our functional-analytic and representation-theoretic techniques apply only to invariant operators on $G$-manifolds, we will in this section describe a more general calculus of proper, uniform pseudodifferential operators on $M$. 

Our calculus is built locally on the usual H\"ormander classes $S^m$ of symbols uniform in the space variable, and the corresponding classes $L^m$ of $\Psi DO$s. That is, 
\begin{definition} Let $U$ be an open set in $\mathbb R^n$ and $m\in\mathbb R$. A function $a\in C^\infty(U\times\mathbb R^n)$ is said to belong to $S^m(U)$ if it has the property that for any compact $K\subset U$ and multiindices $\alpha,\beta$, there exists a constant $C_{K\alpha\beta}$ such that 
\begin{equation}\label{Sm}|D_x^\beta D_\xi^\alpha a(x,\xi)| \le C_{K\alpha\beta} (1+|\xi|)^{m-|\alpha|}, \qquad (x\in U,\ \xi\in\mathbb R^n).\end{equation}
\end{definition}
As usual, a symbol $a\in S^m$ gives rise to an operator $A\in L^m$, $A: C^\infty_c(U)\to C^\infty(U)$, via the iterated integral
\[Au(x) = \int \dbar\xi\int dy\, e^{i(x - y)\cdot\xi}\, a(x,\xi)u(y).\]
%
%%%%%%%%%%%%%%%%%%%
\subsection{Invariant structures on $M$} Here we will construct invariant geometric structures on $M$ with which to define our uniform classes of $\Psi DO$s. Our first claim guarantees that the results of \cite{Ko1} hold in our setting.
\begin{lemma}\label{gequiv} There exists a $G$-invariant Riemannian metric $g$ on $M$ and any two such metrics are equivalent.
\end{lemma}
\begin{proof} Let $(O_k)_1^N$ be an open cover of $X$ such that for every $k$ the $G$-subbundle $G\to \pi^{-1}(O_k)\to O_k$ is trivial. Taking the direct product of a right-invariant metric on $G$ with any metric on $O_k$, we obtain a $G$-invariant metric on $G\times O_k$, hence on $\pi^{-1}(O_k)$. Let $(\phi_k)_1^N$ be a partition of unity on $X$ subordinate to the covering $(O_k)_k$ and lift the $\phi_k$ to obtain an invariant  partition of unity $(\varphi_k)_k$ with $\varphi_k := \phi_k\circ\pi$. Now glue the metrics on the trivial bundles $\pi^{-1}(O_k)$ together with $(\varphi_k)_k$. The equivalence follows from the fact that any $G$-invariant metric is uniquely determined by its restriction to the compact quotient. \end{proof} 
So let us choose an invariant Riemannian structure $g$ on $M$. Denote by $|v|_g$ the length of a tangent or form according to $g$ and by ${\rm dist}_g(p,q)$ the geodesic distance between points $p, q\in M$. The metric we obtain by this method yields a complete Riemannian manifold with bounded geometry; see \cite{R, Ko1}. In particular, the injectivity radius $r_{\rm inj}$ of $M$ is positive.

With respect to $g$, choose a global, $G$-invariant orthonormal frame field for $TM$ and with respect to this, in a ball of radius $r_{\rm inj}$ at each point $p\in M$, define geodesic normal coordinates $(x_1^{(p)}, x_2^{(p)}, \dots, x_n^{(p)})$. For $u\in C^\infty(M)$ put $\partial_j u(p) = \partial_{x_j^{(p)}}u(p)$ and for a multiindex $J=(j_1,j_2,\dots,j_n)$ set 
\[\partial^J u(p) = \partial_1^{j_1}\partial_2^{j_2}\dots\partial_n^{j_n} u(p).\]
Finally, for $j\in\mathbb N$ define 
\[|\partial^j u (p)|_\infty = \max\{|\partial^J u(p)|\, \mid j_1+j_2+\dots+ j_n = j\}.\]
Since $X=M/G$ is compact, we may choose finitely many points $(p_k)_1^N$ in $M$ such that there exist open balls $(U_{p_k})_1^N$ centered at these points with the following properties:
\begin{enumerate}
\item For each $p\in(p_k)_1^N$, the neighborhood $U_p$ has geodesic coordinates $(x_1^{(p)},x_2^{(p)},\dots,x_n^{(p)})$ that extend beyond its closure.
\item The $G$-translates of the union $\bigcup_k U_{p_k}$ cover $M$. 
\end{enumerate}
The action of $t\in G$ on $p\in M$ we write simply $p\mapsto pt$ and we denote by $\rho_t$ the right-translation on functions; $(\rho_t u)(p) = u(pt)$ for $p\in M$, $t\in G$. For $t\in G$, in the neighborhood $U_{pt} := U_p\cdot t$, $p\in (p_k)_1^N$, we thus obtain geodesic coordinates from the translates of the coordinates in $U_p$; 
\[(x_1^{(pt)},\dots,x_n^{(pt)})\quad {\rm with}\quad x_j^{(pt)} := \rho_t x_j^{(p)}.\] 
\subsection{Uniform classes of proper pseudodifferential operators on $M$} These are defined similarly in \cite{MS2} and \cite[\S2]{Ko1}.
\begin{definition}\label{defulg} Let $m$ be a real number. The class $UL^m(M)$ consists of the operators $A$ on $M$ with Schwartz kernel $K_A$ such that 

\begin{enumerate}

\item[(i)] There exists a constant $C_A>0$ such that $K_A(p,q)=0$ whenever ${\rm dist}_g(p,q)>C_A$.

\item[(ii)] $K_A$ belongs to $C^\infty(M\times M\setminus\Delta)$ where $\Delta=\{(p,p)\mid p\in M\}$, and for each $\epsilon>0$ satisfies the estimates 
\[ |\partial^j_p \partial^k_q K_A(p,q)|_\infty \le C_{jk\epsilon}, \qquad ({\rm dist}_g(p,q)\ge\epsilon>0), \]
where $j,k$ are arbitrary, and the subscripts on the derivatives indicate that the derivatives act with respect to the appropriate slot of $K_A$.
\item[(iii)] The family $\{A_{p_k t}: C_0^\infty(U_{p_k t}) \to C^\infty(U_{p_k t})\}$ of restrictions of $A$ to $U_{p_k t}$ forms a collection of operators in $L^m(U_{p_k t})$ for which the bounds in \eqref{Sm}, in terms of the coordinates $(x_j^{(p_k t)})_j$, are uniform with respect to $t\in G$.
\end{enumerate}
Define also the class $UL^{-\infty}(M)$ to be those operators satisfying condition (i) above, but which also obey $K_A\in C^\infty(M\times M)$ and, as in (ii),
\[ |\partial^j_p \partial^k_q K_A(p,q)|_\infty \le C_{jk}, \]
but with no restriction on the distance between $p$ and $q$.
\end{definition}

\begin{rem}\label{3facts} Let us collect some interpretations and consequences.

\begin{enumerate}
\item As usual, the operator $A_p(x^{(p)},D_{x^{(p)}}): C_c^\infty(U_p) \to C^\infty(U_p)$ will be given by its local representations 
\[A_p(x^{(p)}, D_{x^{(p)}})u(x^{(p)}) = \int\dbar\xi\int dy^{(p)}\, e^{i(x^{(p)} - y^{(p)})\cdot\xi} a_p(x^{(p)},\xi)u(y^{(p)}).\]
\item An operator can be pieced together from local representations by a special covering of $M$ of finite multiplicity as in Lemma 3.1 and Prop.\ 3.1 of \cite{MS2}.
\item Condition (iii) gives that, in the coordinates $x^{(pt)} = (x_j^{(pt)})_j$, the operator $A_{pt}$ can be written as a sum 
\[a_{pt}(x^{(pt)}, D_{x^{(pt)}}) + R_{pt}\quad {\rm with} \quad a_{pt}\in S^m(U_{pt}),\quad R_{pt}\in L^{-\infty}(U_{pt}),\]
and the symbols $a_{pt}(x^{(pt)}, \xi)$ satisfy the estimates \eqref{Sm} uniformly in $t\in G$ and $p\in(p_k)_1^N$.
\item It turns out that $UL^{-\infty}(M) = \bigcap_{m\in\mathbb R} UL^m(M)$ and thus these operators have bounded extensions to $L^2$ by Schur's lemma.
\end{enumerate}
\end{rem}
\begin{example}\label{ex1} The tangent bundle of $M$ has a natural decomposition as follows. The {\it vertical space} $V_p\subset T_pM$ consists of those tangents in the kernel of $d\pi:TM\to TX$ and at each point is canonically isomorphic to the Lie algebra $\mathfrak g$ of $G$. A {\it horizontal space} $H_p\subset T_pM$ is a complement to the vertical space. The differential $d\pi$ maps $H_p\subset T_pM$ isomorphically to $T_{\pi(p)}X$, thus a tangent vector in $T_qX$, lifts uniquely to a tangent 
vector in $H_p$ at any $p\in M$ such that $\pi(p)=q$. In particular, we can uniquely lift vector fields of $X$ to horizontal vector fields of $M$ and these lifts can be integrated to yield local sections $X\supset O\hookrightarrow M$. In the presence of an invariant Riemannian metric, we choose $TM = V\oplus H$ where the direct sum is one of $G$-subbundles in $TM$. 

Since $G$ is a Lie group, $TG$ has global frame fields, but it may happen that $TX$ does not. Still, using an invariant partition of unity as in the proof of Lemma \ref{gequiv}, we may choose an orthonormal frame of invariant vector fields $X_1,\dots, X_n$ of the restriction of $TM$ to ${\rm supp}(\varphi_k)$ such that $X_1,\dots, X_d$ span $V\cong\mathfrak g$, and $X_{d+1},\dots, X_n$ span $H$. For any multiindex $J$, the operator $X^J$ is invariant on $M$ of order $|J| = \sum j_k$. As usual, operators built of these objects on each of the sets ${\rm supp}(\varphi_k)$ can be added to form global operators.

An operator $A=\sum_{|J|\le m} a_J X^J$ is in $UL^m(M)$ iff $|X^J a_K |\le C_{JK}$ for any multiindices $J, K$ with $|J|\le m$ and such an operator is $G$-invariant if and only if the functions $a_K$ are constant in the vertical directions; {\it i.e.}\ $X a_K=0$ for $X\in V$. \end{example}
%
%%%%%%%%%%%
\subsection{Properties of the uniform classes} At this point, we will list a collection of properties of the operators in $UL^m(M)$. These follow from the results of \cite[\S2]{Ko1}.  
\begin{prop} If $A\in UL^m(M)$, then its formal adjoint $A^*$ is also in the same class. If $A\in UL^{m_1}(M)$ and $B\in UL^{m_2}(M)$, then $AB\in UL^{m_1+m_2}(M)$. Also, $A\in UL^m(M)$ and $B\in UL^{-\infty}(M)$, imply that $AB, BA\in UL^{-\infty}(M)$.   \end{prop}
The usual $L^2$ continuity of the zero class holds:
\begin{prop}\label{0cont} If $m\le 0$ and $A\in UL^m(M)$, then there exists a $C>0$ such that $\|Au\|_{L^2(M)}\le C\|u\|_{L^2(M)}$ for all $u\in C^\infty_c(M)$. Thus, $A$ can be extended to a bounded linear operator in $L^2(M)$. \end{prop}
Let us now deal with ellipticity in the classes $UL^m(M)$ and the construction of Sobolev spaces. The treatment is identical to that of \cite[\S3]{Ko1}.
\begin{definition} An operator $A\in UL^m(M)$ is said to be \emph{uniformly elliptic} if there exist constants $C_1, C_2, C_3$ such that the symbols $a_{pt}(x^{(pt)},\xi)$ of the operators $A_{pt} =A|_{U_{pt}}$, in the selected coordinates of $U_{pt}$, satisfy
\[C_1 |\xi|^m \le |a_{pt}(x^{(pt)},\xi)| \le C_2 |\xi|^m,\]
uniformly in $U_{pt}$ and for $|\xi|>C_3,\ t\in G,\, p\in (p_k)_1^N$. The class of uniformly elliptic operators in $UL^m(M)$ we will denote by $EUL^m(M)$.\end{definition}
\begin{lemma} \label{lamelambdas} For any $s\in\mathbb R$ there exists an operator $\Lambda_s\in EUL^s(M)$. Furthermore, by taking $P=\Lambda_{s/2}^*\Lambda_{s/2}$ or $\bid +\Lambda_{s/2}^*\Lambda_{s/2}$, we obtain nonnegative and positive operators in the same class.\end{lemma}
Operators in $EUL^m(M)$ possess parametrices:
\begin{prop}\label{badparam} If $A\in EUL^m(M)$, then there exists an operator $B\in EUL^{-m}(M)$ such that 
\[BA = \bid - R_1, \quad AB = \bid - R_2, \quad {\rm with} \quad R_1, R_2\in UL^{-\infty}(M).\]
\end{prop}
In terms of Lemma \ref{lamelambdas}, one could define Sobolev spaces $H^s(M)$ invariant under the group action (but not necessarily on which $G$ acts unitarily); see \cite[\S3]{Ko1}. This invariance and the compactness of $X$ would imply that the $H^s(M)$ would not depend on the definition taken, and thus these are natural objects. We will do better later, but already we can state versions of Sobolev space continuity and elliptic regularity for $EUL^m(M)$: 
\begin{prop}\label{sobodiff} If $A\in UL^m(M)$ and $s\in\mathbb R$, then $A$ extends to a continuous linear operator $A: H^s(M)\to H^{s-m}(M)$. Also, if $A\in EUL^m(M)$, $u\in H^{-\infty}(M)$, and $Au\in H^s(M)$, then $u\in H^{s+m}(M)$.\end{prop}
Prop.\ \ref{sobodiff} is the main tool in demonstrating the following
\begin{prop}\label{esssa} If the operator $A\in EUL^m(M)$ is formally self-adjoint, then it is essentially self-adjoint and its closure in $L^2(M)$ has domain equal to $H^m(M)$.\end{prop}
\begin{rem} The $L^2$ continuity of the zero class (Prop.\ \ref{0cont}) and the Sobolev mapping properties of elliptic operators (Prop.\ \ref{sobodiff}) hold true also in $L^p$ and the corresponding $L^p$--Sobolev spaces, respectively, for $1<p<\infty$. See \cite[\S XI.2]{Ta} for proofs in the case of compact manifolds. For the complete $L^p$ theory, the reader is directed to \cite{Ko1}. \end{rem}
%
%%%%%%%%%%%%%%%%%%%
\section{Hilbert modules and traces}\label{HMT}
%%%%%%%%%%%%%%%%%%%
%
\subsection{Invariant Hilbert space decompositions and von Neumann algebras}\label{hook} In this section we discuss operators $A$ that are invariant under the action of $G$ on $M$. Such operators' Schwartz kernels $K_A$ have the following property:
\begin{equation}\label{sym}K_A(p,q) = K_A(pt,qt), \qquad (p,q\in M,\, t\in G). \end{equation}
Taking a piecewise smooth section $\sigma:X\hookrightarrow M$ and using it to represent points $p\in M$ as pairs $p = \sigma(x)t \leftrightarrow (t,x)\in G\times X$, the relation \eqref{sym} allows us to write 
\begin{equation}\label{conv}\kappa(st^{-1}; x,y) := K_A(\sigma(x)s,\sigma(y)t)=K_A(p,q)\end{equation}
with $s,t\in G$ and $x,y\in X$. Thus $K_A$ descends to a distribution $\kappa$ on the quotient $\frac{M\times M}{G}$ with respect to the quotient measure, here denoted $\frac{dpdq}{dt}$.

Let us for the moment take $M=G$. The algebra of operators $\mathcal L_G\subset\mathcal B(L^2(G))$ commuting with the right action of $G$ is a von Neumann algebra consisting of some left convolutions $\lambda_\kappa$ against distributions $\kappa$ on $G$. We will need the following fact about $\mathcal L_G$. 
\begin{prop}\label{ped}\cite[\S\S 5.1, 7.2]{P} There is a unique normal, faithful, semifinite trace $\trg$ on $\mathcal L_G \subset\mathcal B (L^2(G))$ agreeing with
\[\trg({\lambda_\kappa}^*\lambda_\kappa) = \int_G ds\, |\kappa(s)|^2, \]    
\noindent
whenever $\lambda_\kappa\in\mathcal B(L^2(G))$ and $\kappa\in L^2(G)$. Furthermore, 
$\trg(A^*A)<\infty$ if and only if there exists a $\kappa\in L^2(G)$ for which 
$A=\lambda_\kappa\in\mathcal B(L^2(G))$.  
If we define $\tilde\kappa(t) = \bar\kappa(t^{-1})$, and if $\kappa_k, \mu_k\in L^2(G)$, $k=1,\dots, N$, then the operator $A= \sum_1^N \lambda_{\tilde \kappa_k} \lambda_{\mu_k}$ belongs to $\dom(\trg)$. Furthermore, $A$ takes the form $A=\lambda_\kappa$ for some continuous $\kappa$ and $\trg(\lambda_\kappa) = \kappa(e)$.\end{prop}
\begin{rem}The unimodularity of $G$ is necessary for the trace property of $\trg$. \end{rem}
In order to bring the trace on $\mathcal L_G$ up to the manifold, we will need the following ideas. For any (complex) Hilbert space $\mathcal H$ define a \emph{free Hilbert $G$-module} as $L^2(G)\otimes\mathcal H$. The action of $G$ in $L^2(G)\otimes\mathcal H$ is defined by $G\ni t\mapsto \rho_t\otimes\bid$. A general Hilbert $G$-module is a closed $G$-invariant subspace in a free Hilbert $G$-module. 

With the smooth action of $G$ and invariant Riemannian density $dp$ on $M$, by fixing a Haar measure $dt$ on $G$, we obtain a finite quotient measure $dx$ on $X=M/G$. With this, we may present the Hilbert $G$-module $L^2(M)$ in the form
\[ L^2(M, dp)\cong L^2(G, dt)\otimes L^2(X, dx),\]
which makes it a free Hilbert $G$-module. It follows that we have
a decomposition of the von Neumann algebra of bounded invariant operators
\[ \mathcal B(L^2(M))^G\cong 
\mathcal B(L^2(G))^G\otimes \mathcal B(L^2(X))\cong 
\mathcal L_G\otimes \mathcal B(L^2(X)),\] 
where we have made the identification $\mathcal L_G\cong \mathcal B(L^2(G))^G$. The von Neumann algebra $\mathcal L_G \subset \mathcal B (L^2(G))$ possesses a unique trace $\trg$ as we have mentioned in Prop.\ \ref{ped}. In order to measure the invariant subspaces of $L^2(M)$, we will need a trace on $\mathcal L_G\otimes\mathcal B(L^2(X))$, which can be constructed as follows. If $(\psi_l)_{l\in\mathbb N}$ is an orthonormal basis for $L^2(X)$, then we have the decomposition
\begin{equation}\label{decompfortrans}L^2(M) \cong L^2(G)\otimes L^2(X) \cong \bigoplus_{l\in \mathbb N} L^2(G)\otimes \psi_{l}.\end{equation}
Denoting by $P_m$ the projection onto the $m^{th}$ summand, we obtain 
a matrix representation of $A\in\mathcal B(L^2(M))$ with elements $A_{lm} := P_l A P_m \in\mathcal B(L^2(G))$. If $A\in\mathcal B(L^2(M))^G$, then these matrix elements are bounded, invariant operators in $L^2(G)$ and so there exist distributions $\kappa_{lm}$ on $G$ so that $A\in\mathcal B(L^2(M))^G$ has a matrix representation
\[A \leftrightarrow [A_{lm}]_{lm}=[\lambda_{\kappa_{lm}}]_{lm}.\]
\begin{definition} For positive $A\in\mathcal B(L^2(M))^G$ define 
\[\Trg (A) = \sum_{l\in\mathbb N}\,\trg(A_{ll}).\]
\end{definition}
The functional $\Trg$ is a normal, faithful, and semifinite trace and is independent of the basis $(\psi_l)_l$ used in its construction, {\it cf.}\ \cite[\S V.2]{T}. Analogously to the classical case, define the $G$-{\it Hilbert-Schmidt} operators in terms of $\Trg$ by
\[\dom_{1/2}(\Trg)=\{A\in \mathcal B(L^2(M))^G\mid \Trg(A^*A)<\infty\}\]  
and define the $G$-{\it trace-class} by
\[\dom(\Trg)=\{C=\sum_{k=1}^N A_k^*B_k \mid A_k, B_k\in\dom_{1/2}(\Trg)\},\]
where $N$ depends on $C$.
%
%%%%%%%%%%%%%%%%%%%
\subsection{Membership in $\dom_{1/2}(\Trg)$}\label{big}
%%%%%%%%%%%%%%%%%%%
%
In this section we prove Prop.\ \ref{bigprop} and apply it to obtain that for $\epsilon>0$, we have $UL_G^{-n/2-\epsilon}(M)\subset\dom_{1/2}(\Trg)$. We begin with a calculation taken verbatim from \cite{P1} which we repeat for the convenience of the reader.
\begin{lemma}\cite[Lemma 6.3]{P1} Let $A\in \mathcal B(L^2(M))^G$ with $K_A\in L^2(\frac{M\times M}{G})$. It follows that, in terms of the expression \eqref{conv}, we have
\[\Trg(A^*A)=\|\kappa \|_{L^2(G\times X\times X)}^2 = \int_{\frac{M\times M}{G}}\frac{dpdq}{dt}\, |K_A(p,q)|^2.\]
\end{lemma}
\begin{proof} Let $(\psi_k)_k$ be an orthonormal basis for $L^2(X)$.  In the decomposition $L^2(M)\cong \bigoplus_k L^2(G)\otimes\psi_k$, the invariant operator $A$ has a matrix representation $A\to [\lambda_{\kappa_{kl}}]_{kl}$.  In terms of this, we compute
{\small \begin{align}\Trg(A^*A) &=\sum_l \trg((A^*A)_{ll})=\sum_l\trg\left(\sum_k(A^*)_{lk}A_{kl}\right)\notag\\
&=\sum_l\trg\left(\sum_k A^*_{kl}A_{kl}\right)=\sum_{kl}\trg(\lambda^*_{\kappa_{kl}}\lambda_{\kappa_{kl}})=\sum_{kl}\|\kappa_{kl}\|_{L^2(G)}^2\notag\end{align}}
\noindent
by normality of $\trg$.  

Now, except on a set of measure zero we may take $p= \sigma(x)t \leftrightarrow(t,x)$ and obtain a description of $A$ as in \eqref{conv}
\begin{align}(Au)(p)=&\int_M dq\, K_A (p,q)u(q)\notag\\
=&(Au)(t,x)=\int_{G\times X} ds dy\ \kappa(s;x,y)u(st,y).\notag\end{align}
The distributional kernels $\kappa_{kl}$ can be recovered from $\kappa$ by projecting into the summands in $L^2(M)\cong \bigoplus_l (L^2(G)\otimes\psi_l)$,
\[\kappa_{kl} = \int_{X\times X} \, dxdy\, \kappa(\,\cdot\, ;x,y)\psi_l (y)\bar\psi_k (x).\]
Let us compute the norm of $\kappa$ in $L^2(G\times X\times X)$. Since $(\psi_k)_k$ is an orthonormal basis for $L^2(X)$, the set $(\bar\psi_k\otimes \psi_l)_{kl}$ forms an orthonormal basis for $L^2(X\times X)$. By construction, $\kappa_{kl}$ is equal the $kl^{th}$ Fourier coefficient of $\kappa$ with respect to the decomposition $L^2(G\times X\times X)\cong\bigoplus_{kl}(L^2(G)\otimes\psi_k \otimes \psi_l)$. Hence 
\[\sum_{kl} \|\kappa_{kl}\|^2_{L^2(G)}=\|\kappa \|_{L^2(G\times X\times X)}^2,\]
\noindent
which is the result. The last assertion, that $\|\kappa \|_{L^2(G\times X\times X)} = \|K_A\|_{\frac{M\times M}{G}}$, follows from the definitions.\end{proof}
The previous lemma can be applied when the image of an operator is smooth enough.
\begin{prop}\label{point} Let $n=\dim M$. If $s>n/2$ and $A\in\mathcal B(L^2(M))^G$ has $\im(A)\subset H^s(M)$, then $K_A\in L^2(\frac{M\times M}{G})$.\end{prop}
 \begin{proof} Since $A$ is defined on all of $L^2(M)$ and is into $H^s(M)$, the closed graph theorem implies that it is continuous. Since $s>n/2$, the Sobolev lemma provides that $\im(A)\subset C_0(M)$ and the existence of a constant $C$ such that
\begin{equation}\label{key}\sup_{p\in M} |(Au)(p)|\le C \|Au\|_{H^s(M)} \le C\|A\|_{L^2\to H^s}\, \|u\|_{L^2},\end{equation}
uniformly for $u\in L^2(M)$. Fixing a $p\in M$, this estimate implies that the Riesz representation theorem can be applied, providing a function $k_p\in L^2(M)$ so that for any $u\in L^2(M)$,
\[(Au)(p) = \langle k_p,u\rangle_{L^2},\]
and furthermore, $\sup_{p\in M}\| k_p\|_{L^2}\le C\|A\|_{L^2\to H^s}$.

Now, $K_A$ is a Schwartz kernel, but since $(Au)(p)=\int_M K_A(p,q)u(q)dq$ and agrees with $\langle k_p,u\rangle$ when $u\in C_c^\infty(M)$ we have $k_p=K_A(p,\ \cdot \ )$ at every $p\in M$. 

Denoting the $t$-translate of $p$ by $pt$, 
\begin{align}\|k_{pt}\|_{L^2}^2 &= \int_M dq\, |K_A(pt, q)|^2  = \int_M dq\, |K_A(p, qt^{-1})|^2 \notag\\
&=  \int_M dq\, |K_A(p, q)|^2 =\|k_p\|_{L^2}^2\notag\end{align}
by invariance of $A$ and the measure. Denoting by $x$ a representative of $p$ in $M/G$ and by $\mu$ the quotient measure on $X$, the compactness of $X$ together with the bound on $\|k_p\|_{L^2}$ imply that 
\[\int_X d\mu(x)\, \|k_{x}\|^2_{L^2} \le \mu(X)\, C^2 \|A\|_{L^2\to H^s}^2.\]
But $\int_X d\mu(x)\, \|k_{x}\|^2_{L^2} = \|K_A\|_{L^2(\frac{M\times M}{G})}^2$.\end{proof}
Prop.\ \ref{bigprop} follows by concatenating the preceding assertions.
\begin{corollary}\label{suffGHS} Let $A$ be a $G$-invariant operator in $UL^m(M)$ with $m<-n/2$, $\dim M = n$. It follows that $\Trg(A^*A)<\infty$.\end{corollary}
\begin{proof} By Prop.\ \ref{sobodiff}, $A:L^2(M)\to H^m(M)$ and so Prop.\ \ref{bigprop} applies.\end{proof}
\subsection{Anisotropic Sobolev embedding and the transversal dimension} The technique from the previous section answers a natural question regarding restrictions and function space embeddings in which the vertical and horizontal directions have different character. The results of this section do not depend on the unimodularity of $G$, but strengthen the results of the previous section when $G$ is unimodular. For applications, the reader is directed to \cite{P3}.

For $s\in\mathbb R$, denote by $H^{s,\epsilon}(M)$ the Hilbert space tensor product $H^s(G)\otimes H^\epsilon(X)\subset L^2(G)\otimes L^2(X)$. Choosing a section $\sigma$ and invariant partition of unity $(\varphi_k)_k$ as above, the space $H^{s,0}(M)$ can be defined as the completion of $C^\infty_c(M)$ in the norm given by
\[\|u\|_{H^{s,0}}^2 = \int_X dx\, \|u(\cdot,x)\|_{H^s(G)}^2.\]
\begin{lemma} For $2s>\dim G$, it is true that $H^{s,0}(M)$ is contained in the space $BW_0(G, L^2(X))$ of bounded, weakly continuous functions from $G$ to $L^2(X)$ which vanish at infinity. \end{lemma}
\begin{proof} Observe that $u\in H^{s,0}(M)$ implies that $u$ is the kernel of a Hilbert-Schmidt operator $A: L^2(X)\to H^s(G)$:
\[(Af)(t) = \int_X dx\, u(t,x) f(x).\]
Since the Hilbert-Schmidt norm majorizes the operator norm, we have $\|A\|_{L^2(X)\to H^s(G)} \le \|u\|_{H^{s,0}}$. With $2s>\dim G$, the point evaluations $G\ni t\mapsto (Af)|_t$ are well-defined and so $u$ gives rise to a map $G\times L^2(X)\ni (t,f)\longmapsto (Af)(t) \in\mathbb C$. Now fix $t\in G$ and note that the Sobolev lemma gives
\begin{align}|(Af)(t)| &\le \sup_{t'\in G}|(Af)(t')|\lesssim \|Af\|_{H^s(G)} \notag\\
& \le \|A\|_{L^2(X)\to H^s(G)}\|f\|_{L^2(X)}\le \|u\|_{H^{s,0}}\|f\|_{L^2(X)},\notag\end{align}
so, as before, we obtain the existence of an element $v_t\in L^2(X)$ such that $(Af)(t) = \langle v_t, f\rangle_{L^2(X)}$ for $f\in L^2(X)$ and $\|v_t\|_{L^2(X)}\le \|u\|_{H^{s,0}}$. Furthermore, $v_t = u(t,\cdot)$. Now, for each $f\in L^2(X)$, we have
\[\langle v_t-v_{t'}, f\rangle_{L^2(X)} = |(Af)(t)-(Af)(t')| \longrightarrow 0\]
as $t\to t'$ because the Sobolev lemma gives that $Af$ is continuous, thus $G\ni t\longmapsto u(t,\cdot)\in L^2(X)$ is weakly continuous. Vanishing at infinity follows from Fubini.\end{proof}
\begin{corollary}\label{transvdim} With $2s>\dim G$ and $\epsilon>0$, suppose that a self-adjoint projection $P\in\mathcal B(L^2(M))^G$ has $L=\im(P)\subset H^s(G)\otimes H^\epsilon(X)$. It follows that there exists a finite-dimensional subspace $V\subset L^2(X)$ such that $\im(P)$ is included in the space of continuous functions from $G$ to $V$.\end{corollary}
\begin{proof} A small variation on the proof of \cite[Thm. 4.2]{P3} obtains $\im(P)\hookrightarrow BW_0(G, V)$, but since the norm and weak topologies coincide on $V$, the space $BW_0(G, V)$ consists of norm continuous functions. \end{proof}
The preceding statement ties our discussion here to \cite{P3} and strengthens the result there. The minimal $\dim_{\mathbb C}V$ which we could call the transversal dimension ``$\dim_X$'' of $\im(P)$ leads to a generalized Fredholm theory of its own.
%
%%%%%%%%%%%%%%%%%%
\section{The $G$-trace class of invariant operators}\label{gtraceandmsc}
%%%%%%%%%%%%%%%%%%
%
\subsection{Invariant properly supported operators} Here we will define the main object of study in this paper. 
\begin{definition} For the classes $UL^m(M)$, $EUL^m(M)$, we indicate the subclass of invariant elements of any of the above classes by including a subscript $G$.  \end{definition}
\begin{rem}  Since $K_A(p,q) = \kappa(st^{-1};x,y) = 0$ whenever ${\rm dist}_g(p,q) = {\rm dist}_g(\sigma(x)s,\sigma(y)t)$ exceeds a constant $C_A$, properly supported invariant operators are integrations against distributions with compact support in $\frac{M\times M}{G}$.\end{rem}
As we have chosen an invariant Riemannian structure $g$ as in Lemma \ref{gequiv}, its associated Laplace-Beltrami operator $\Delta_g$ is invariant as well. Thus $\Delta_g\in EUL_G^2(M)$ and this affords us a refined version of Lemma \ref{lamelambdas}: 
\begin{lemma}\label{goodlambdas} For any $s\in\mathbb R$, there exists an invariant operator $\Lambda_s\in EUL_G^s(M)$.\end{lemma}
\begin{proof} Denoting by $|\xi|_g$ the length of $\xi\in T^*M$ according to the Riemannian structure $g$, the symbol $a(\xi) = |\xi|_g^{s/2}$ is $G$-invariant since $g$ is and the corresponding operator $\Lambda_s$ belongs to $EUL^s(M)$. Since $g$ is a polynomial in the components of $\xi$, it follows that $\Lambda_s$ is a classical operator. Local operators associated to $a$ can be patched together as in \cite[Thm.\ 5.1]{Sh}.\end{proof}
\begin{rem} As in Lemma \ref{lamelambdas}, we may construct nonnegative and positive operators in the same class.\end{rem}
%
%%%%%%%%%%%%%%%%%%%
\subsection{Connes-Moscovici averaging}\label{cma} The article \cite{CM} contains an averaging technique by which, given an ordinary compactly supported pseudodifferential operator, one obtains a $G$-invariant one. This will be the bridge between the general uniform classes and the invariant objects that our von Neumann algebraic formalism can handle. We will describe this technique here and give some consequences. 

If $A$ is a pseudodifferential operator with compact support, then for each $u\in C^\infty_c(M)$, the average 
\[{\rm Av}(A)u  = \int_G dt\, \rho_t A \rho_{t^{-1}}\, u\]
makes sense, defining a smooth function on $M$, as it only involves integration over a compact subset of $G$. By this method, properly supported $G$-invariant pseudodifferential operators can be obtained from compactly supported ones. 

An important application of the averaging method for us is the invariant version of Prop.\ \ref{badparam}:
\begin{prop}\label{goodpram} If $A\in EUL_G^m(M)$, then there exists an operator $B\in EUL_G^{-m}(M)$ such that 
\begin{equation}\label{goodz} BA = \bid - R_1, \quad AB = \bid - R_2, \quad {\rm with} \quad R_1, R_2\in UL_G^{-\infty}(M).\end{equation}
\end{prop}
\begin{proof} We begin with the ordinary properly supported parametrix, the existence of which is given by Prop.\ \ref{badparam} and apply \cite[Prop.\ 1.3]{CM}. This depends on the existence of a cutoff function for $M$; see \cite[p295]{CM}. With the bundle map $\pi$, our invariant partition of unity $(\varphi_k)_1^N$, a piecewise smooth section $\sigma:X\to M$, and a cutoff function $\psi$ on $G$, define $t_{k,p}\in G$ so that $p=\sigma(\pi(p))t_{k,p}$ holds in the support of $\varphi_k$. Then the function $f(p)=\sum \varphi_k(p) \psi(t_{k,p})$ possesses the required properties. \end{proof}
Connes and Moscovici give a version of Prop.\ \ref{esssa}.
\begin{prop}\label{cms} With $m\ge 1$, consider $A\in EUL_G^m(M)$ as an operator in $L^2(M)$ with dense domain $C_c^\infty(M)$. It follows that the domain of the closure of $A$ coincides with the subspace of all $u\in L^2(M)$ for which $Au\in L^2(M)$, in the distributional sense. \end{prop}
\begin{proof} The regularization procedure from the proof of \cite[Prop.\ 3.1]{A} works here as the main ingredient is the invariant parametrix from Prop.\ \ref{goodpram}. The exhaustion method from the proof of \cite[Lemma 3.1]{CM} remains valid using the cutoff functions again from Prop.\ \ref{goodpram}.\end{proof}
\begin{rem} Prop.\ \ref{cms} allows us to reinterpret the formal adjoint $A^*$ of an operator $A$ satisfying its hypotheses as the Hilbert space adjoint as well. Also, from now on we will always consider the $L^2$-closures of pseudodifferential operators by convention.\end{rem}
\begin{prop}\label{corellfin} Let $A\in EUL_G^m(M)$. It follows that $\dim_G\ker(A)<\infty$. \end{prop}
\begin{proof} Let $B\in EUL_G^{-m}(M)$ be a parametrix for $A$. With $R_1=\bid-BA\in UL_G^{-\infty}(M)$, we see that $u\in\ker(A)$ satisfies $R_1u=u$; thus $\ker(A)\subset\im(R_1)\subset H^\infty(M)$. Since $A$ is closed, there is an invariant, self-adjoint projection $P$ onto $\ker(A)$, so we have $\im(P)\subset H^\infty(M)$. But $\Trg(P) = \Trg(P^* P)<\infty$ by Prop.\ \ref{bigprop}. \end{proof}
Since taking the adjoint of an operator preserves proper support and ellipticity, we could follow Connes and Moscovici at this point and go on to define the $G$-index of operators in $EUL_G^m(M)$. This is the index claim, Cor.\ \ref{easy} from the introduction.

We here provide a weaker sufficient condition for membership in $\dom(\Trg)$ than used in the proof of Prop.\ \ref{corellfin}. 
\begin{prop}\label{smo} For $r\in\mathbb R$, suppose that $A\in UL_G^m(M)$ for $m< -n$. It follows that $A$ is a $G$-trace-class operator.\end{prop}
\begin{proof} Take  $m< -n$, let $\Lambda=\Lambda_{-m/2}\in EUL_G^{-m/2}(M)$ be the operator constructed in Lemma \ref{goodlambdas} and let $B\in EUL^{m/2}_G(M)$ be $\Lambda$'s parametrix so that $B\Lambda =\bid-R$ with $R\in UL^{-\infty}_G(M)$. Multiplying this equation through by $A$ we get 
\[A=B\Lambda A+RA.\] 
Since $\Lambda A, B\in UL_G^{m/2}(M)$, we have $\Lambda A, B\in\dom_{1/2}(\Trg)$ by Cor.\ \ref{suffGHS} which also implies that $A,R\in\dom_{1/2}(\Trg)$.\end{proof}
%
%%%%%%%%%%%%%%%%%%%%%
\section{Existence theory for pseudodifferential equations}\label{PW}
%%%%%%%%%%%%%%%%%%%%%
%
\subsection{The $G$-Fredholm property} In this section we will give an application of our $G$-trace result to solving problems in $L^2$ involving invariant operators on $G$-bundles. See \cite[\S\S1,2]{Srr} and \cite[\S3]{P1} for much finer results on $G$-morphisms and the $G$-Fredholm property.
\begin{definition}\label{gfred} Let $\mathcal H_1, \mathcal H_2$ be Hilbert spaces on which $G$ acts strongly continuously and unitarily. A closed, densely defined, $G$-invariant operator $A:\mathcal H_1\to\mathcal H_2$ is said to be $G$-{\it Fredholm} if $\dim_G\ker(A)<\infty$ and if there exists a closed, invariant subspace $Q\subset\im(A)$ so that $\dim_G(\mathcal H_2\ominus Q)<\infty$. \end{definition}
\begin{prop}\label{ellfred} Let $m\ge1$ and assume that $A\in EUL_G^m(M)$ is self-adjoint. It follows that $A$ is $G$-Fredholm. \end{prop}
\begin{proof} The property that $\dim_G\ker(A)<\infty$ is given by Prop.\ \ref{corellfin}. For the second, we will form the space $Q$ as a spectral subspace of $A$. To that end, let $A=\int_{-\infty}^\infty \lambda dE_\lambda$ be the spectral resolution of $A$ and put $P_\delta = \int_{-\delta}^\delta dE_\lambda$ for $\delta\ge 0$. Now, Prop.\ \ref{esssa}  gives that the domain of $A$, as an operator in $L^2$, is precisely $H^m(M)$. Further, $A$ preserves $\im(P_\delta)\subset\dom(A)=H^m(M)$, thus $A^k u\in H^m(M)$ for $u\in\im(P_\delta)$ as well. But then Prop.\ \ref{sobodiff} implies that $\im(P_\delta)\subset H^\infty(M)$. Now continue as in the proof of Prop.\ \ref{corellfin} and take $Q=\im(P_\delta)^\perp$. \end{proof}
\begin{rem} The operator $A$, when restricted to $Q=\im(P)^\perp$, has a bounded inverse in $L^2$. In symbols, $\|A^{-1}|_Q\|_{L^2\to L^2}<\infty$. The elliptic estimate then implies that solutions on this subspace gain $m$ degrees on the Sobolev scale.\end{rem}
\subsection{Example} L.\ H\"ormander showed in \cite{H2} that if the iterated commutators of a collection of vector fields generate the tangent space, then a quadratic form constructed with them satisfies a subelliptic estimate. There is a version of this assertion with a simpler proof, as follows.
\begin{prop}\label{horm}\cite[Thm.\ 5.4.7]{FK} Suppose $X_1,\dots, X_m$ are complex vector fields on the real manifold $M$ such that each $\bar X_j$ is a linear combination of the $X_j$. Suppose also that the iterated brackets $X_j, [X_{j_1}, X_{j_2}], [X_{j_1}, [X_{j_2}, X_{j_3}]], \dots$ of order $\le p$ span all vector fields on $M$. Then if $V$ is a relatively compact subdomain of $M$, 
\begin{equation}\label{qgains}\|u\|_{2^{1-p}}^2 \lesssim \sum_1^m \|X_j u \|_{L^2}^2 + \|u\|_{L^2}^2\end{equation}
uniformly for all smooth functions $u$ supported in $V$. \end{prop}
Take $M=G$, with $G$ a unimodular Lie group with a chosen Haar measure and with $(X_j)_j$ right-invariant vector fields on $G$ satisfying the hypotheses above. Define the (positive) Hermitian form
\[Q(u,v) =  \sum_1^m \langle X_j u, X_j v\rangle_{L^2}  + \langle u,v\rangle_{L^2}.\]
Noting that the Hilbert space adjoint of an invariant vector field $X$ with respect to $L^2(G)$ is given by $X^*=-X$, the Hermitian form $Q$ is associated to the operator $A = -\sum_1^m X_j^2 + \bid$. Clearly $A$ is formally self-adjoint and invariant, and though $A$ is not elliptic, it is true that $A$ is essentially self-adjoint, \cite[\S3]{DGSC}. The estimate \eqref{qgains} expresses the fact that $Q$ gains Sobolev degree, so the techniques of \cite{KN} and \cite{P1} provide that $A$ is $G$-Fredholm.

One could concoct richer examples by noting that a cocompact, normal Lie subgroup $N\subset G$ gives rise to a fibration $N\to G\to X$ on which our formalism is applicable. Homogeneous spaces are the subject of \cite{CM}. 

For classes of pseudodifferential operators related to $A$ above, the reader is directed to \cite[Ch.\ XV]{Ta}.

%%%%%%%%%%%%
\subsection{The generalized Paley--Wiener theorem of Arnal--Ludwig} Using the $G$-Fredholm property to derive existence results relies on the following easy fact.  
\begin{lemma}\label{easy2} Let $L$ be a Hilbert $G$-module and let $L_1, L_2\subset L$ be Hilbert submodules such that $\dim_G L_1>\dim_G (L\ominus L_2)$. It follows that $\dim_G L_1\cap L_2 \ge \dim_G L_1-\dim_G(L\ominus L_2)$. In particular, $L_1\cap L_2\neq \{0\}$.\end{lemma}
Thus, once it has been established that an operator has a good inverse on the complement $Q$ of a subspace of finite $G$-dimension (like $Q=\im(P_\delta)^\perp$ above), it only remains to understand how large subspaces can come about. 

Here we will describe one method in \cite[\S 3]{P2} for determining that a closed, invariant subspace of $L^2(M)$ have infinite $G$-dimension. It is based on a generalized Paley--Wiener theorem of \cite{AL}, which reads
\begin{prop}\label{AL}\cite[Thm.\ 1.3]{AL} Let $G$ be a locally compact, unimodular group with Haar measure $m$, containing a closed, noncompact, connected subset. Let $f$ be in $L^2(G, m)$ such that $m({\rm supp}(f)) < m(G)$ and such that there exists $\kappa\in L^2(G, m)$ with $\lambda_\kappa f = f$. Then $f = 0$ m-a.e. \end{prop}
We immediately can apply this fact to our case.
\begin{corollary}\label{AL2} If a closed, right-invariant subspace $L\subset L^2(G)$ contains a nonzero function with compact support, then $\dim_G L=\infty$.\end{corollary}
\begin{proof} Let $f$ be such a function. The projection $P=\lambda_\kappa$ onto $L$ then satisfies $Pf=\lambda_\kappa f = f$, implying $\dim_G L = \|\kappa\|^2_{L^2(G)}=\infty$.\end{proof}
On the $G$-bundle $M$, we have a global translation by elements of $G$, thus we may define $\langle f\rangle$ to be the smallest closed, invariant subspace of $L^2(M)$ containing $f$. In symbols, 
\[\langle f\rangle = \overline{\left\{\sum_{k=1}^N \alpha_k\, \rho_{t_k}f\mid \alpha_k\in\mathbb C ,\ t_k\in G,\ N<\infty \right\}},\] 
where the closure is in $L^2(M)$. If $f\in L^2(M)$ has compact support, Cor.\ \ref{AL2} extends to yield $\dim_G\langle f\rangle = \infty$:
\begin{corollary}\label{bigPW}\cite[Cor.\ 2.3]{P2} Let $G\to M\to X$ be a principal $G$-bundle with $G$ a unimodular Lie group. If $0\neq f\in L^2(M)$ has compact support, then $\dim_G \langle f\rangle = \infty$.\end{corollary}
It follows that if $A$ is a $G$-Fredholm operator and $f\in L^2(M)$ with compact support, then $Au=g$ has good $L^2$ solutions for $g\in\langle f\rangle$ orthogonal to a finite-$G$-dimensional subspace of $\langle f\rangle$. The question remaining is whether any of the solutions is interesting. For example, it easy to see that $f\in C^\infty_c(\mathbb R)$ generates $\langle f\rangle = L^2(\mathbb R)$. In order to say anything useful about solving differential equations, we will need to construct subspaces of $\langle f\rangle$ consisting of smooth functions. 
\begin{prop}\label{exhaust}\cite[\S3]{P2} If $f\in C^\infty_c(M)$, then there exist closed, invariant subspaces $\langle f\rangle_\delta\subset\langle f\rangle$, $\delta>0$, such that 
\begin{enumerate}
\item $\langle f\rangle_\delta\subset H^s(M)$ for all $s\in\mathbb R$, 
\item $\dim_G \langle f\rangle_\delta\to\infty$ as $\delta\to 0^+$,
\item For any $\delta>0$, the elements $g$ of $\langle f\rangle_\delta$ are of the form $g=\rho_\kappa f$ with $\kappa\in C^\infty\cap L^2(G)$.
\end{enumerate}
\end{prop}
\begin{proof} First consider the case in which $M=G$ and define $\langle f\rangle_\delta$ as follows. Let $\rho_f = U|\rho_f|$ be the polar decomposition of $\rho_f$ and write the spectral decompositon $|\rho_f|=\int_0^C\lambda dE_\lambda$. Further, for $\delta\in[0,C]\cup\{0^+\}$, put $P_\delta=\int_\delta^C dE_\lambda$ and define
\begin{equation}\label{deldef}\langle f\rangle_\delta=\{\rho_\kappa f\mid\kappa\in\im(P_\delta)\}.\end{equation}
For any $\delta>0$, $\rho_f$ is boundedly invertible on $\im(P_\delta)$ so the composition $\kappa\mapsto\rho_f \kappa\mapsto\rho_\kappa f$ is a (reversing) $G$-isomorphism from $\im(P_\delta)$ to $\langle f\rangle_\delta$. In particular, the spaces have identical $G$-dimensions. By the normality of $\trg$ and the fact that the $P_\delta$ are a spectral family, we need only establish that $\trg(P_{0^+})=\infty$ in order to establish property (2). But $U\,\im(P_{0^+}) = \overline{\im(\rho_f)}$ which contains $f$, which has compact support, so Cor.\ \ref{bigPW} gives the result. Observing that for $\delta>0$ we have $\im(P_\delta)\subset\im(\rho_{\tilde f}\rho_f)\subset C^\infty(G)$, we obtain the third claim. The first follows from this and Young's inequality. The case for a principle bundle $M$ is derived from the previous observations by applying them to a nonzero Fourier coefficient of $f\in C^\infty_c(M)$ in the decomposition \eqref{decompfortrans}. Thus we obtain a family of kernels $\kappa\in\im(P_\delta)$ as before and we define $\langle f\rangle_\delta$ as in Eq.\ \eqref{deldef}. \end{proof}
\begin{rem} In \cite[\S2.2]{P2} it is also shown that a closed, invariant subspace $L\subset L^2(M)$ containing an element $f$ such that ${\rm esssup}\,|f| = \infty$ but ${\rm esssup}\,\{\|f(\cdot, x)\|_{L^2(G)}\mid x\in X\}<\infty$ also has $\dim_G L=\infty$.
\end{rem}
%
%%%%%%%%%%%%%%%%%%
\subsection{Proof of the main theorem} We have now all the ideas needed to derive our principal result.
Prop.\ \ref{ellfred} gives that $A$ is $G$-Fredholm. Prop.\ \ref{exhaust} provides that for $\delta>0$ sufficiently small, the spectral projection $P_\delta$ of $A$ satisfies $\im(P_\delta)^\perp\cap\langle f\rangle_\delta\neq\{0\}$. The regularity results follow from the elliptic estimate Prop.\ \ref{sobodiff}. By Prop.\ \ref{goodpram}, we have an invariant parametrix $B$ such that $ABf = (\bid - R_2)f$. Now $L=\{f\mid f=R_2f\}$ is $L^2$-closed in $\im(R_2)\subset H^\infty(M)$ and thus $\dim_G L<\infty$ and so cannot contain any element with compact support. But $\ker(B)\subset L$ so it must not either. \qedsymbol
\begin{rem} We point out that in general, the parametrix solution to $Au=f$ with $f\in C_c^\infty(M)$ might give $Bf=u=0$ and the error term $R_2f$ precisely $f$. But the solution constructed here cannot exhibit this behavior, as the last statement of the main theorem provides.\end{rem}
%
%%%%%%%%%%%%%%%%%%%%%%%
\ack We thank Giuseppe Della Sala, Bernhard Lamel, Elmar Schrohe, and Alex Suciu for helpful conversations and the Math Department of Leibniz Universit\"at Hannover for its generous hospitality and lively discussions.
%%%%%%%%%%%%%%%%%%%%%%%
%

%
\end{document}